\documentclass[12pt]{article}
\usepackage[utf8]{inputenc}
\usepackage{amsmath,amssymb,amsthm}
\usepackage{hyperref}
\usepackage{graphicx}
\usepackage{xcolor}
\renewcommand\le{\leqslant}
\renewcommand\ge{\geqslant}
\newcommand\eps{\varepsilon}
\newcommand\R{\mathbb R}
\newcommand\E{\mathsf E}
\renewcommand\P{\mathsf P}
\DeclareMathOperator\rank{rank}
\DeclareMathOperator\sign{sign}
\DeclareMathOperator\diam{diam}
\DeclareMathOperator\conv{conv}

\usepackage[normalem]{ulem}

\newtheorem{theorem}{Theorem}
\newtheorem{statement}{Statement}

\newtheorem*{corollary}{Corollary}

\theoremstyle{definition}

\newtheorem*{remark}{Remark}

\title{Polynomial approximation on disjoint segments and amplification of approximation}
\author{Yu.~Malykhin, K.~Ryutin}

\begin{document}

\maketitle

\abstract{
We construct explicit easily implementable  polynomial approximations of
sufficiently high accuracy 
for locally constant functions on the union of disjoint segments (see~\eqref{def_en}).  This problem has important applications  in
several areas of numerical analysis,
complexity theory, quantum algorithms, etc. The one, most relevant for us, is
the amplification of approximation method: it allows to construct
approximations of higher degree $M$ and better accuracy from the approximations
of degree $m$.}

\section{Introduction}

\paragraph{The problem setting.}
Let $I_1,\ldots,I_s$ be closed disjoint intervals on the real
line, and $K :=
I_1\sqcup \ldots \sqcup I_s$. We introduce the following quantity:
\begin{equation}
\label{def_en}
    E_n(I_1,\ldots,I_s) := \max_{y_1,\ldots,y_s\in[-1,1]} E_n(\sum_{i=1}^s
    y_i\mathbf{1}_{I_i})_K,
\end{equation}
where $E_n(f)_{K} := \min\limits_{P\in\mathcal P_n}\|f-P\|_{C(K)}$, and
$\mathbf{1}_I$ is the indicator function of $I$. 

In other words, $E_n(I_1,\ldots,I_s)$ is equal to the smallest $\eps$, such
that for any values $y_i$, $|y_i|\le 1$, there exists a real
algebraic polynomial $P$ of degree at most
$n$, such that $|P(t)-y_i|\le \eps$, $t\in I_i$, $i=1,\ldots,s$.

We  also consider the following modification of $E_n$:
\begin{equation}
\label{def_en_bnd}
    E_n^*(I_1,\ldots,I_s) := \max_{y_1,\ldots,y_s\in[-1,1]}
    E_n^*(\sum_{i=1}^s y_i\mathbf{1}_{I_i})_K,
\end{equation}
where $E_n^*(f)_K := \min \{\|f-P\|_{C(K)}\colon P\in\mathcal P_n,\;\|P\|_{C(\conv
K)}\le \|f\|_{C(K)}\}$, i.e. the approximating polynomial is bounded (in
absolute value) on $I_i$ and between them by $\max|y_i|$.

It is obvious that
$$
E_n(I_1,\ldots,I_s) \le E_n^*(I_1,\ldots,I_s).
$$

Our goal is to obtain explicit and rather general upper bounds on $E_n$ and
$E_n^*$ that can be applied in different situations.

\paragraph{Amplification of approximation.}
The main motivation for us is the development of the so-called  amplification of approximation technique. 
Let us fix the set of values $Y=\{y_1<y_2<\ldots<y_s\}$ and let
$0<\delta<\frac12 \min|y_{i+1}-y_i|$. Suppose that there exists a polynomial
$P\in\mathcal P_n$ such that
\begin{equation}
    \label{poly_prop}
    |P(t)-y_j|\le \eps\quad\mbox{for all $t$: $|t-y_j|\le \delta$}. 
\end{equation}
If $f$ is a function with values in the set $Y$ and $g$
approximates $f$ with an error not exceeding $\delta$, then $P\circ g$ will approximate $f$ with
an error $\le\eps$. We denote the smallest $\eps$ such that~\eqref{poly_prop} holds for some
$P\in\mathcal P_n$ as $\eps_n(Y;\delta)$.

If $L$ is any $m$-dimensional linear space of functions and $q\in\mathcal P_n$,
then $\{q\circ f\colon f\in L\}$ lies in
a linear space of dimension at most
\begin{equation}
    \label{M_value}
    M := \sum_{j=0}^n\binom{m+j-1}{j}.
\end{equation}
We remark that $M\le (m+1)^n$.

In fact a similar property implies that $q$
should be a polynomial. 
Namely, if $q\colon\R\to\R$ is measurable and for any linear space $L$
of functions of fixed dimension $m\ge 2$ the dimension of $\mathrm{span}\,\{q \circ f:
f\in L\}$ is bounded by some $C(m)$, then $q$ is a polynomial
(see~\cite{Pink} for the precise result).

We give several examples of the amplification of approximation by
linear spaces or $m$-term approximations. The essence of this
technique is to construct approximations of higher degree $M$
(see~\eqref{M_value}) and better accuracy $\eps := \eps_n(Y,\delta)$ from the
approximations of degree $m$ with accuracy $\delta$.

1. Kolmogorov widths.
We recall the definition of the Kolmogorov width of a set $K$ in a normed space
$X$:
$d_m(K,X) := \inf\limits_{L_m} \sup\limits_{x\in K} \inf\limits_{y\in L_m} \|x-y\|_X$,
where the infimum is taken over linear subspaces of $X$ of dimension at most
$m$.

Let $\mathcal F$ be some class of functions with values in $Y$, then
$$
d_m(\mathcal F,L_\infty)\le \delta\quad\Longrightarrow\quad
d_M(\mathcal F,L_\infty)\le\eps.
$$

2. Approximate rank.
There is a related notion of the approximate rank of a matrix:
$\rank_\eps(A):=\min\{\rank B\colon \max_{i,j}|A_{i,j}-B_{i,j}|\le\eps\}$.

For any matrix $A$ with entries $A_{i,j}\in Y$ we have
    $$
    \rank_\delta(A) \le m \quad\Longrightarrow\quad \rank_\eps(A)\le M.
    $$
Note that elementwise product of rank-one tensors has also rank one, so
    the amplification technique also works for tensors. Let
        $T=(T_{i_1,\ldots,i_d})$ be a $Y$-valued tensor, then
    $$
    \rank_\delta(T) \le m \quad\Longrightarrow\quad \rank_\eps(T)\le M.
    $$

3. $m$-term approximations. We recall that
the best $m$-term approximation by a dictionary $D\subset X$ is defined as
$\sigma_m(x,D)_X := \inf\limits_{\substack{g_1,\ldots,g_m\in D\\c_1,\ldots,c_m\in\R}}
\|x-\sum_{k=1}^m c_k g_k\|_X$.

Let $D$ be some dictionary of functions closed under pointwise 
    multiplication, and let $f$ be a function taking values in $Y$. Then
        $$
        \sigma_m(f,D)_{L_\infty} \le \delta\quad\Longrightarrow\quad
        \sigma_M(f,D)_{L_\infty} \le \eps.
        $$

\paragraph{Properties of $E_n$ and $\eps_n$.}
There is
an obvious relation between $\eps_n$ and $E_n$:
\begin{equation}
    \label{eps_en}
\eps_n(Y;\delta) \le E_n([y_1-\delta,y_1+\delta],\ldots,[y_s-\delta,y_s+\delta])\cdot \frac12|y_s-y_1|.
\end{equation}
Indeed, an appropriate  translation
allows us to deal with the case
$-y_1=y_s=\frac12|y_s-y_1|$ and approximate the function with values $y_j/y_s$
via $E_n([y_1-\delta,y_1+\delta],\ldots,[y_s-\delta,y_s+\delta])$.

The function $y\mapsto E_n(\sum_{i=1}^s y_i\mathbf{1}_{I_i})_K$ is convex,
therefore the maximum in~\eqref{def_en} is attained at some extremal point of
the cube $[-1,1]^s$, i.e. $y\in\{-1,1\}^s$. Similar arguments work for $E_n^*$.

In particular, for $s=2$ we have $E_n(I_1,I_2)=E_n(-\mathbf{1}_{I_1}+\mathbf{1}_{I_2})_{I_1\sqcup
I_2}$ and the inequality~\eqref{eps_en} becomes equality, e.g.
$$
\eps_n(\{-1,1\};\delta) = E_n([-1-\delta,-1+\delta],[1-\delta,1+\delta]).
$$
Moreover, one can use the approximation amplification argument and obtain the bound
\begin{equation}
    \label{en_amplif}
E_{mn}(I_1,\ldots,I_s) \le \eps_m(\{-1,1\}; E_n(I_1,\ldots, I_s)).
\end{equation}
Indeed, it is enough to construct approximation of degree $mn$ for any
function $f=\sum_{i=1}^s y_i\mathbf{1}_{I_i}$, where
$y_1,\ldots,y_s\in\{-1,1\}$. We can take the composition of optimal polynomials
for the quantity $\eps_m$ and the $n$-degree polynomial for $f$.

The analogue of~\eqref{en_amplif} for $E_n^*$ is
\begin{equation}
    \label{en_amplif_bnd}
E_{mn}^*(I_1,\ldots,I_s) \le E_m^*([-1,-1+\eps],[1-\eps,1]),\quad
\eps := E_n^*(I_1,\ldots,I_s).
\end{equation}

\paragraph{Historical remarks.}

The origins  of our problems can be traced to the   classical results by  E.I.
Zolotar\"{e}v, N.I. Akhiezer  on the best  uniform approximation by  rational or
polynomial  functions on the union of two disjoint segments.  Various important
results and methods of rational and polynomial approximations were developed by
many prominent mathematicians including  G.~Szeg\"o, G.~Faber, J.~Walsh, E.~Stiefel,
A.A.~Gonchar. These problems are important for electrical engineering: a very
interesting direction of research is the optimization of multiband electrical
filters (see the survey~\cite{Bogat19} by A.~Bogatyr\"ev for the state of the
art and the references). The main problem of this direction is to study rational
approximations to a piecewise constant function that takes only two values.
   
Polynomial approximations on disjoint segments or more generally
on disjoint Jordan curves  were actively studied since the 60s.  H.~Widom in his
seminal paper~\cite{Widom} developed the necessary machinery for the problem
of polynomial approximation: he studied Chebyshev polynomials for a system of
disjoint Jordan curves in $\mathbb{C}$. 
We would like to mention the papers~\cite{Bogat99},\cite{Schinorm21},\cite{Schi2interv} and references therein where one can learn about  different aspects of this important line of investigations.

W.H.J.~Fuchs used the approach developed by Widom and   showed
\cite{Fuchs80} 
that $E_n(f)$ behaves like $C n^{-1/2} \exp(-\alpha n)$ for functions taking
constant values on the system of disjoint segments on $\mathbb{R}$ where
$\alpha>0$ depends on the geometry of these  intervals. 
  In fact different quantities relevant for the order of
 approximation (such as $\alpha$) or even
the constructions of operators of nearly best polynomial approximation  are
determined by the Green function of the complementary domain to the union of
these segments in $\mathbb{C}$. It is not easy to compute the Green function or
evaluate different quantities that are needed  for these  approximation
problems. Hasson obtained the exact exponent for the
union of $s$ equal segments rotated by some angle $2\pi/s$,
see~\eqref{Hasson}.
We should  mention the papers~\cite{Andr,ET,ShStrWat,Hass} where the
reader can find necessary references concerning some applications and further
developments of this approach.  We will discuss important contributions by
A.~Eremenko and P.~Yuditskii later.

There are also very interesting related problems such as parametric
approximation of piecewise analytic functions~\cite{Kon90} and fast decreasing
polynomials~\cite{IvaTot90}. Several constructions of polynomials that approximate
locally constant functions are given in these papers.

Several authors independently used polynomials for the approximation
amplification.

E.D.~Gluskin (see~\cite{Gluskin}) used the amplification technique in order to get
new lower bounds for the Kolmogorov widths of octahedra in $\ell_\infty^N$. The
structure of this problem allowed him to use simple polynomial $p(t)=t^n$,
that ``amplifies'' approximation of the value $y_1=0$ and
preserves the value $y_2=1$.
Later this result as well as the approach were re-discovered by
N.~Alon~\cite{Alon}, some interesting
applications were also given in this paper.

A.~Klivans and A.~Sherstov~\cite{KS10} used the amplification technique for the
approximation of boolean functions (in fact, Kolmogorov widths of these
classes) and for the approximation rank of boolean (signum) matrices. They used
rather crude polynomial approximation of the function $\sign(t)$ for
$1-\delta\le|t|\le 1+\delta$.

On the other hand several problems studied by the computer science community
also lead to the problem of explicit polynomial approximation
for functions taking a finite number of values (especially boolean functions).

Sherstov \cite{Sherrobust} showed that every bounded polynomial on the Boolean
hypercube can be made robust to noise in the inputs with only a
constant-factor increase in degree. That is just the problem of polynomial
approximation of some  function $g$  that takes constant values on the system
of cubes of side length $2/3$ with centers at all the points from $\{0,1\}^d$
(in fact $g$ agrees with a polynomial on $\{0,1\}^d$). 
This work was motivated by an earlier paper~\cite{BNRW} where authors justified
the importance of robust polynomials for the quantum algorithms.
Although the approximation takes place in the $d$-dimensional
space, the core of their method uses univariate polynomial approximation of the
signum function. See also~\cite{ShHarSubsp}.

There are also more recent results on quantum computations that
also make use the approximation of signum, see~\cite{LinTong} and references
therein.

We also mention some applications of this topic to the
problems on principal component projections and stable matrix
approximations (see \cite{Matr}, \cite{Matr2}). The boundedness of approximating
polynomials is important here (hence the notion of $E_n^*$ is relevant).

\paragraph{Known results.}
If follows from the mentioned results of Fuchs that
$$
A_1n^{-\frac12-c}\exp(-\alpha n) \le E_n(I_1,\ldots,I_s) \le 
A_2n^{-\frac12}\exp(-\alpha n),
$$
for any disjoint segments in $\R$; here $A_1,A_2$ and $\alpha$ depend on the
geometry of the segments. (We believe that the extra factor $n^{-c}$ is not
required when we approximate locally constant functions.)

The results of Hausson~\cite{Hass} imply that for the following system of
rotated segments in $\mathbb{C}$ we have
\begin{equation}
    \label{Hasson}
\limsup_{n\to\infty}
    E_n(I_1,\ldots,I_s)^{1/n}=\sqrt[s]{\frac{b^{s/2}-a^{s/2}}{b^{s/2}+a^{s/2}}},\quad
    I_k := e^{2\pi i \frac{k}{s}}[a,b].
\end{equation}

The case of $E_n$ for two segments boils down to the
polynomial approximation of the sign function:
$E_n(\sign t)_{I_1\sqcup I_2} = E_n(I_1,I_2)$, where $I_1$ and $I_2$ are
separated by zero.
The asymptotics for the best approximation
on two
symmetric segments was given by Eremenko and Yuditskii~\cite{ErYu07}. The case of
two arbitrary segments was done in \cite{ErYu11}. We formulate their result
in the symmetric case. Let $K=[-A,-1]\cup [1,A]$, $A>1$, then we
have the asymptotics for
$L_n:=E_n(\sign t)_K=E_n([-A,-1],[1,A])$:
\begin{equation}
\label{2segm_asymp}
L_{2m+2}=L_{2m+1}\sim \frac{\sqrt{2}(A-1)}{\sqrt{\pi A}}(2m+1)^{-1/2}\left(\frac{A-1}{A+1}\right)^m.
\end{equation}

The well--known results of K.G. Ivanov and V. Totik~\cite{IvaTot90} give the following bounds:
\begin{multline}
    \label{totik_sign}
    \exp(c_1n\log(1-\alpha)) \le E_n([-1,-\alpha],[\alpha,1]) \le \\
    \le E_n^*([-1,-\alpha],[\alpha,1]) \le \exp(c_2n\log(1-\alpha)),
\end{multline}
provided that $n\ge C_1$ and $n\log(1-\alpha)\le -C_2$.

\paragraph{Notation.}
In order to provide explicit bounds we use the following quantities determined by the set of segments
$I_1,\ldots,I_s$:
\begin{itemize}
    \item number of segments $s$,
    \item maximum ``radius'' $\delta := \frac12 \max |I_i|$,
    \item minimum distance between segments: $\sigma := \min\limits_{i\ne j}
        \min\limits_{\substack{x\in I_i\\y\in I_j}}|x-y|$,
    \item the diameter $D := \diam (I_1\cup\ldots\cup I_s)$.
\end{itemize}
We will obtain estimates on $E_{n}(I_1,\ldots,I_s)$ in terms of
$s,\delta,\sigma,D$ and we estimate the $\eps_n(Y;\delta)$ quantity in terms of
$s,\delta$, $\hat{\sigma}:=\min_{i\not=j}|y_i-y_j|$, $\hat{D}:=\diam Y$.

\paragraph{Our contribution.}
We use three classical methods of approximation: Bernstein polynomials, Jackson
theorem and Lagrange interpolation in Newton form. Therefore we provide
explicit approximating polynomials and give bounds valid for all $n$ for the
approximation error ($E_n$ and $E_n^*$) with 
explicit dependence on the geometry of the system $I_1,\ldots,I_s$.

There are at least two interesting asymptotic regimes  for the
quantity $E_n(I_1,\ldots,I_s)$:
\begin{enumerate}
    \item[I.] $\sigma/D$ is small, i.e. some of the segments are very close;
    \item[II.] $\delta/D$ is small, i.e. when all the
        segments are small (the dependence on the number of segments is important in this case).
\end{enumerate}

Case I. 
In Section~\ref{sec_main} we obtain the upper bound
$$
E_n^*(I_1,\ldots,I_s)\le2\exp(-cn\sigma/D).
$$
This agrees with the bound~\eqref{totik_sign} for two segments $I_1$, $I_2$ of
equal length: indeed, we have $\sigma/D=\alpha$ and if this quantity is bounded away from $1$, then ~\eqref{totik_sign} gives
$$
\exp(-c_1n\sigma/D)\le E_n(I_1,I_2) \le E_n^*(I_1,I_2) \le \exp(-c_2n\sigma/D).
$$

Note also that the asymptotics~\eqref{2segm_asymp} when 
$\sigma/D=1/A$ is small (and fixed) gives
$$
L_n\asymp n^{-1/2}\exp(-c_1 n\sigma/D),\quad c_1\approx 1.
$$

Case II. Let us look at the asymptotics~\eqref{2segm_asymp} when 
$A=1+2\delta$, where $\delta>0$ is small (but fixed). Then
\begin{equation}
\label{2segm_smalldelta}
L_n \asymp n^{-1/2}(c_2\delta)^{n/2},\quad c_2\approx 1.
\end{equation}
We prove the upper bound that
generalizes this to the case $s>2$:
$$
    E_n(I_1,\ldots,I_s) \le
    \left(C(s,u)\frac{\delta}{D}\right)^{\lfloor\frac{n+1}{s}\rfloor},
    \quad\mbox{if $\sigma/D\ge u>0$.}
$$
Note that this bound is sharp in the following sense. If we replace $s$ by a (fixed)
number $r<s$ then the inequality does not hold  for sufficiently small
$\delta$; namely, $\delta/D\le c(r,s,u)$.

Therefore our estimates capture the right dependence on the parameters in these cases
and generalize known results.

As a result, we achieve amplification of the approximation for
functions taking an arbitrary (finite) set of values.
We hope that our estimates can be applied  to some problems in approximation theory and computer science that were described above.

\paragraph{Acknowledgments.}
The starting point for our research was the relevant
property of the Bernstein polynomials that we learned from Vladimir
Sherstyukov: these polynomials provide good approximations on the segments where the
function is linear. We would like to thank the anonymous referee for his/her precise comments  and suggestions which improved the presentation.

This research was carried out at Lomonosov Moscow State University, Moscow Center of Fundamental and Applied Mathematics  with the financial support of the Russian Science Foundation (Grant No. 22- 11-00129).

\section{Approximation by Bernstein polynomials}

We make use of the Bernstein polynomials
$$
B_n(f,x)=\sum_{k=0}^n \binom{n}{k} f(k/n)x^k(1-x)^{n-k}. 
$$
Note that $B_n(f,x)=\E(f(\xi_{n,x}/n))$, where $\xi_{n,x}$ is the random variable 
with the binomial $B(n,x)$ distribution, i.e. $P(\xi_{n,x} =
m)=\binom{n}{m}x^m(1-x)^{n-m}$.

We mention two properties of the Bernstein operators that are important for us: $B_n$ is
positive ($B_nf\ge 0$ if $f\ge0$), and $B_n$ preserves constants.

For $p,q\in[0,1]$ we recall the definition of the divergence function
$$
H(p\|q) :=  p\ln\frac{p}{q}+(1-p)\ln\frac{1-p}{1-q}.
$$

\begin{statement} 
    \label{stm_berni}
    Let $f$ be a function defined on $[0,1]$ such that $|f|\le M$.
    Assume that $f\equiv \mathrm{const}$ on an interval $(a,b)$, $0<a<b<1$.
    Then for any
    $x\in(a,b)$ and $n\in\mathbb N$ we have:
    \begin{equation}
        \label{berni_stmt}
        |B_n(f,x)-f(x)| \le 2M(e^{-nH(a\|x)}+e^{-nH(b\|x)}).
    \end{equation}

    If $f\equiv\mathrm{const}$ on $[0,b)$ then the first term is missing and for
    any $x\in[0,b)$ and $n\in\mathbb N$ we have:
    $$
        |B_n(f,x)-f(x)| \le 2Me^{-nH(b\|x)}.
    $$
    Analogously for the case of interval $(a,1]$.
\end{statement}

\begin{proof}
Let us consider the case $0<a<b<1$.
Since Bernstein operator preserves constants, we may assume that $f\equiv0$ on
$(a,b)$ and $\|f\|_\infty\le 2M$.

    Let $x\in (a,b)$.
We estimate $|B_n(f,x)|=|\E(f(\xi_{n,x}/n))|$
simply as $2M(\P(\xi_{n,x} \le an) + P(\xi_{n,x} \ge bn))$.

    Let us bound the probability $\P(\xi_{n,x} \le an)$
    using the Chernoff--Hoeffding~\cite{Hf63} bound:
$$
    \P(\xi_{n,x} \le an) \le \exp(-nH(a\|x)).
$$
We deal with $P(\xi_{n,x} \ge bn)$ in a similar way. 

    In the case when $f\equiv\mathrm{const}$ on $[0,b)$ we
    have to deal only with $P(\xi_{n,x} \ge bn)$.
\end{proof}

\begin{remark}
    Using sharp probability estimates by Serov, Zubkov~\cite{SZ12}, one can obtain for
    fixed $0<a<x<b<1$ that
    $$
    |B_n(f,x)-f(x)| \le C(a,b,x)Mn^{-1/2}(e^{-nH(a\|x)}+e^{-nH(b\|x)}).
    $$
    More precise inequality was announced in~\cite[Theorem 20]{TSP}.
\end{remark}

\paragraph{The case of two segments.}

From Statement~\ref{stm_berni} one can immediately get a bound for
$E_n^*([0,h],[1-h,1])$, $0<h<1/2$. Indeed,
    if we wish to approximate $y_1\mathbf{1}_{[0,h]}+y_2\mathbf{1}_{[1-h,1]}$, we can
    apply the Statement to the function $f=y_1 \mathbf{1}_{[0,1/2)} +
    y_2\mathbf{1}_{(1/2,1]}$. For the polynomial $P := B_nf$ we have the bound
    $$
    \|f-P\|_{C([0,h]\bigcup [1-h,1])} \le 2\exp(-nH(\frac12\|h)) =
    2(4h(1-h))^{n/2}.
    $$
It is obvious that $|P|\le\max(|y_1|,|y_2|)$ on $[0,1]$.
Hence
\begin{equation}
    \label{h_bound}
E_n^*([0,h],[1-h,1]) \le 2(4h(1-h))^{n/2}.
\end{equation}

\begin{corollary}
   For any disjoint segments $I_1,I_2$ of equal length we have
    \begin{equation}
        \label{two_segments_bound}
    E_n^*(I_1,I_2) \le
    2\left(4\frac{2\delta}{D}(1-\frac{2\delta}{D})\right)^{n/2}
        = 2\left(1-\frac{\sigma^2}{D^2}\right)^{n/2}.
    \end{equation}
\end{corollary}

\begin{proof}
We place the segments
inside $[0,1]$ by some suitable translation and multiplication by $D^{-1}$.
We arrive to the setting of the approximation on
    $[0,\frac{2\delta}{D}]\cup[1-\frac{2\delta}{D},1]$ and
    apply~\eqref{h_bound}. The equality in~\eqref{two_segments_bound} holds
    since $D=\sigma+4\delta$.
\end{proof}

\begin{remark}
We can reduce the case of different lengths to the considered one by taking
the smaller segment and enlarging it while keeping $\delta$ and $\sigma$ fixed,
so that the resulting diameter is $D'=\sigma+4\delta$. Then we
apply~\eqref{two_segments_bound} for $\delta$ and $D'$.
\end{remark}

In the case of the approximation amplification for $\{-1,1\}$-valued functions we
have $I_1=[-1-\delta,-1+\delta]$, $I_2=[1-\delta,1+\delta]$, $D = 2+2\delta$, so
$\eps_n(\{-1,1\};\delta) \le 2(4\delta/(1+\delta)^2)^{n/2}$.
Similarly for the $\{0,1\}$-valued functions we have
$\eps_n(\{0,1\};\delta) \le (8\delta/(1+2\delta)^2)^{n/2}$.

The bound~\eqref{two_segments_bound} is poor when $\sigma/D$ is small.
Indeed, it behaves like $\exp(-n(\sigma/D)^2)$, but Theorem 1 (see the next Section) gives an upper bound $2\exp(-cn\sigma/D)$.  We can
compare estimate~\eqref{two_segments_bound} with the best approximation in
the other  regime (case II): when $\delta/D$ is small.

The estimate \eqref{two_segments_bound} gives the following:
$$
E_n([-A,-1],[1,A])\le 2\left(\frac{(A-1)(A+1)}{A^2}\right)^{n/2}.
$$
Let $A=1+2\delta$, where $\delta>0$ is small (but fixed). We recall that the
asymptotics~\eqref{2segm_smalldelta} behaves like $(c_2\delta)^{n/2}$, 
$c_2\approx 1$, whereas our estimate behaves like $\asymp (c\delta)^{n/2}$,
$c\approx 4$.

We conclude that the Bernstein polynomials provide explicit, nonasymptotic
estimates, but they are fairly good only when $\delta$ is small.

\section{Main results}
\label{sec_main}

\begin{theorem}
    \label{th_general}
    For any system of disjoint segments we have
    $$E_n^*(I_1,\ldots,I_s) \le 2\exp(-cn\sigma/D).$$
    In fact, one can take $c=1/30$.
\end{theorem}

We have a similar general bound for the $\eps_n (Y,\delta)$. Namely, for any
finite $Y\subset \mathbb{R}$ with characteristics $\hat{D},\hat{\sigma}$ we
apply the last estimate and \eqref{eps_en} and get
$$
\eps_n (Y,\delta)\le \frac12\hat{D}\cdot E_n([y_1-\delta,y_1+\delta],\ldots,[y_s-\delta,y_s+\delta])\le
\hat{D}\exp(-cn\sigma/D),
$$
where
$\sigma=\hat{\sigma}-2\delta$ and $D=\hat{D}+2\delta\le 2\hat{D}$. Finally,
$$
\eps_n (Y,\delta)\le \hat{D}\exp(-c_1n(\hat{\sigma}-2\delta)/\hat{D}).
$$ 

\begin{proof}
    Let $D=2$ and $I_1,\ldots,I_s\subset[-1,1]$. First of all we  use the
    Jackson--Favard inequality (see~\cite{DLbook}, p.219)
    $$
    E_n(f)_{[-1,1]} \le \frac{\pi}{2(n+1)}\|f'\|_\infty,
    $$
    that holds for absolutely continuous functions. Let $y_i\in\{-1,1\}$
    (recall the the maximum in~\eqref{def_en_bnd} it attained at such values).
    We consider the piecewise linear continuous extension $f$ of the function $\sum_i y_i\mathbf{1}_{I_i}$ to the
    segment $[-1,1]$ such that $\|f'\|_\infty \le 2/\sigma$.  Pick some
    $0<\tau<1$. We have $\|(1-\tau)f'\|_\infty \le 2(1-\tau)/\sigma$ and
    $\|(1-\tau)f-P\|_{C[-1,1]} \le \tau$ for some polynomial $P$ of degree at most
    $n_0 := \lfloor \pi(1-\tau)/(\sigma\tau)\rfloor$.
    Note that $\|P\|_{C[-1,1]}\le 1$ and $\|f-P\|\le 2\tau$.

    We obtain that
    $E_{n_0}^*(I_1,\ldots,I_s) \le 2\tau$.
    The bound~\eqref{two_segments_bound} gives
    $E_m([-1,-1+2\tau],[1-2\tau,1])\le 2\cdot (4\tau(1-\tau))^{m/2}$.
    If we combine this with~\eqref{en_amplif_bnd}, we will obtain
    that $E_{mn_0}^* \le 2\cdot (4\tau(1-\tau))^{m/2}$.

    Let $n\ge n_0$. Then $mn_0 \le n < m(n_0+1)$ for some $m$ and
    $$
    E_n^* \le E_{mn_0}^* \le 2\cdot (4\tau(1-\tau))^{m/2} \le 2 \cdot
    (4\tau(1-\tau))^\frac{n}{2(n_0+1)}.
    $$
    We have
    $$
    n_0+1 \le \pi(1-\tau)/(\sigma\tau) + 1 =
    \frac{\pi(1-\tau)+\sigma\tau}{\sigma\tau} \le
    \frac{\pi(1-\tau)+2\tau}{\sigma\tau},
    $$
    so,
    $$
    E_n^* \le 2\exp(\frac{n}{2}\cdot
    \frac{\sigma\tau}{\pi(1-\tau)+2\tau}\cdot\ln(4\tau(1-\tau)) =
    2\exp(-c\frac{n\sigma}{2}).
    $$
    In order to optimize constant, we take $\tau=0.15$, then $c=0.034\ldots >
    1/30$.

    Finally, if $n<n_0$, then the right side of the inequality being proved is at least
    $$
    2\exp(-\frac1{30}\; \frac{n_0 \sigma}{2}) \ge 2\exp(-\frac{1}{30}\cdot
    \frac{\pi(1-0.15)}{0.15}\cdot\frac12) > 1.
    $$

    The case of arbitrary $D$ follows after a suitable linear transform.
\end{proof}

Our proof is a generalization of the arguments from~\cite{DGJSV}.

\paragraph{The case of small $\delta$.}

In the case of small $\delta$ we may expect better bounds than $\exp(-cn)$, e.g.
for two segments the dependence on $\delta$ and $n$ is like $\delta^{n/2}$. A similar estimate is  true for $s>2$.

\begin{theorem}
\label{generalC_theorem}
    For any set of points $\mathcal Z = \{z_1,\ldots,z_s\}\subset\mathbb C$ with $|z_i-z_j|\ge 1$, complex numbers $w_1,\ldots,w_s$,
    $|w_i|\le 1$, and $n\in\mathbb N$, there is a complex polynomial $P\in
    \mathcal{P}_{ns-1}^{\mathbb C}$, such that
    for any $\delta < 1/2$ we have
$$
    \max_{i=1,\ldots,s} \max_{|z-z_i|\le\delta} |P(z)-w_i| \le
    2(s-1)(A(\mathcal Z)\delta)^{n},
$$
    where $A(\mathcal Z) := 2(1+2\diam Z)^{s-1}$.
    If all $w_i,z_i \in \mathbb{R}$, then the polynomial $P$ is in fact from $
    \mathcal{P}_{ns-1}$, i.e. is real.
\end{theorem}

Note that this polynomial does not depend on $\delta$.

\begin{corollary}
    \label{cor_delta}
    Let $s\in\mathbb N$, $u>0$. For any system of disjoint segments
    $I_1,\ldots,I_s$ with $\sigma/D \ge u$
    we have
    \begin{equation}
        \label{small_delta_En}
    E_{n}(I_1,\ldots,I_s) \le
    \left(C(s,u)\frac{\delta}{D}\right)^{\lfloor\frac{n+1}{s}\rfloor},
    \quad\mbox{$n\in\mathbb N$.}
    \end{equation}
 \end{corollary} 
 
\begin{proof} Let $m=\lfloor\frac{n+1}{s}\rfloor$, so $ms-1\le n$.
If $m=0$, then the inequality for $E_n$ is obvious; let $m\ge1$.
W.l.o.g. $D=1$; hence $\sigma\ge u$. We shall prove that
$E_{ms-1}(I_1,\ldots,I_s) \le (C(s,u)\delta)^m$.

    We  have $
    E_{ms-1}(I_1,\ldots,I_s)=E_{ms-1}(\tilde{I}_1,\ldots,\tilde{I}_s)$ where
    $\tilde{I}_j:=I_j/u$.  We apply the Theorem  to the points
    $z_j$~--- the
    centers of the segments $\tilde{I}_j$ and arbitrary  $w_j\in [-1,1]$ and
    write
    $$
    E_{ms-1}(\sum_{i=1}^s w_i\mathbf{1}_{\tilde{I}_i})_{\tilde{I}_1\sqcup\ldots \sqcup \tilde{I}_s}\le
    2(s-1)\left(A(\{z_j\}_{j=1}^s)\tilde{\delta}\right)^m,
    \quad\mbox{if $\tilde\delta<1/2$,}
    $$
   where $\tilde{\delta}=\delta/u$. The corollary easily follows
   from this estimate when $\tilde\delta < 1/2$. The
   opposite case is handled by taking a sufficiently large $C$.
  \end{proof}  
Theorem \ref{generalC_theorem} implies  also an  estimate for the $\eps_n$--quantity.
\begin{corollary}
    Let $s\in\mathbb N$, $u>0$. For any set $Y=\{y_1,\ldots,y_s\}\subset\R$
    with $\hat\sigma/\hat D \ge u$ we have
$$
\eps_n(Y;\delta) \le
    \frac{\hat D}{2}\left(C(s,u)\frac{\delta}{\hat D}\right)^{\lfloor\frac{n+1}{s}\rfloor},
    \quad\mbox{$n\in\mathbb N$.}
$$
\end{corollary}

The only tool we need to prove Theorem is the Lagrange polynomial in the Newton
form. Given a function $f$ and a sequence of different points $(\zeta_i)_{i=1}^N$, there
exists a unique polynomial $P\in\mathcal P_{N-1}^{\mathbb C}$ such that
$P(\zeta_i)=f(\zeta_i)$, $i=1,\ldots,N$. Moreover (see~\cite{DLbook}, p.120) we have,
\begin{multline}
    \label{newton}
    P(z) = f[\zeta_1] + f[\zeta_1,\zeta_2](z-\zeta_1) + f[\zeta_1,\zeta_2,\zeta_3](z-\zeta_1)(z-\zeta_2) + 
\ldots \\ + f[\zeta_1,\ldots,\zeta_N](z-\zeta_1)\cdots(z-\zeta_{N-1}).
\end{multline}
The divided differences may be defined recursively as $f[z]:=f(z)$, and 
\begin{equation}
    \label{divdiff}
f[\zeta_1,\ldots,\zeta_k] :=
    \frac{f[\zeta_2,\ldots,\zeta_k]-f[\zeta_1,\ldots,\zeta_{k-1}]}{\zeta_k-\zeta_1}.
\end{equation}

Let us proceed to the proof of  Theorem \ref{generalC_theorem}.
\begin{proof}
    It is enough to construct a sequence of polynomials $P_1,\ldots,P_s$ with two properties:
    \begin{equation}\label{prop1}
    P_1+P_2+\ldots+P_s \equiv 1,
    \end{equation}
    \begin{equation}\label{prop2}
         \max_{|z-z_j|\le\delta}|P_i(z)| \le \eps,\quad\mbox{for all $i\ne j$}.
    \end{equation}
    Then $|P_i(z)-1|\le (s-1)\eps$ if $|z-z_i|\le\delta$. Given
    $\{w_i\}$, we take $P := \sum_{i=1}^s w_i P_i$ and obtain
    $$
    |P(z)-w_i| \le |w_i|\cdot|P_i(z)-1| + \sum_{j\ne i}|w_j|\cdot|P_j(z)| \le
    2(s-1)\eps,\quad\mbox{if $|z-z_i|\le\delta$.}
    $$

    Let $f_i$ be the function such that $f_i(z)=1$ on the disk $|z-z_i|<1/2$ and
    $f_i(z)=0$ on  other disks $|z-z_j|<1/2$, $j\ne i$. We
    could define the polynomial
    $P_i$ simply as the interpolation polynomial for $f_i$ at points
    $z_1,\ldots,z_s$ with multiplicity $n$. However, to avoid multiplicities we
    will take $n$
    different points in the $\gamma$-neighborhood of each $z_j$, interpolate
    $f_i$ at these points and take the limit as $\gamma\to0$.

    Fix some small $\gamma>0$. For each $i=1,\ldots,s$ we take $n$ different points
    $z_i^m$, $|z_i^m-z_i|\le\gamma$, $m=1,\ldots,n$.
    Let $Q_i\in\mathcal P_{ns-1}^{\mathbb C}$ be the interpolation polynomial for the function $f_i$ at the set
    of $ns$ points $\{z_i^m\}$.  The property~\eqref{prop1} for $\{Q_i\}$
    follows from the fact that $Q_1+\ldots+Q_s$ is a polynomial of degree at
    most $ns-1$ and attains value $1$ at $ns$ points.

    Let us bound, say, $|Q_s(z)|$. We arrange the
    interpolation points into the sequence
    $$
    (\zeta_1,\ldots,\zeta_{ns}) :=
    (z_1^1,z_1^2,\ldots,z_1^n,z_2^1,\ldots,z_2^n,\ldots,z_s^n).
    $$
        Note that the points are divided into two groups:
        $\{\zeta_k\colon k \le n(s-1)\}$, where $f_s(\zeta_k)=0$, and
        $\{\zeta_k\colon k > n(s-1)\}$, where $f_s(\zeta_k)=1$.

    Using~\eqref{newton}, we obtain
    \begin{equation}
        \label{explicit}
    Q_s(z) =
    \sum_{k=1}^{ns}f_s[\zeta_1,\ldots,\zeta_k]\prod_{j=1}^{k-1}(z-\zeta_j).
    \end{equation}
    It is clear that the divided differences in~\eqref{explicit}
    are zero for $k \le n(s-1)$.

    Let us show that
    $|f_s[\zeta_{j+1},\ldots,\zeta_{j+k}]|\le 2^{k-1}(1-2\gamma)^{1-k}$
    for all (admissible) $j,k$ using  induction on $k$. The case $k=1$ is
    obvious; let $k>1$. If $f_s$ is constant
    on $\{\zeta_{j+1},\ldots,\zeta_{j+k}\}$, then $f_s[\zeta_{j+1},\ldots,\zeta_{j+k}]=0$. Otherwise, this divided difference is
    a fraction~\eqref{divdiff} with the denominator that is at least $1-2\gamma$
    (since $\zeta_{j+1}$ and $\zeta_{j+k}$ are from
    different groups; recall also that
    $|z_i-z_j|\ge 1$), so we can use the bound for $k-1$.

    Hence, for
    $|z-z_i|\le\delta$ with some $i<s$, we can bound the
    sum as
    \begin{multline}
        \label{bnd}
        |Q_s(z)| \le \sum_{n(s-1)<k\le ns} 2^{k-1}(1-2\gamma)^{1-k}(\delta+\gamma)^n
        (D+\delta+\gamma)^{k-1-n} \le \\
        \le (1-2\gamma)^{-ns}2^{ns}(\delta+\gamma)^n(D+1/2+\gamma)^{n(s-1)}.
    \end{multline}
    (Here $D:=\diam\mathcal Z$.)

    Let us define $P_i$ as the limit of $Q_i$ as $\gamma\to0$. The limit exists
    since in~\eqref{explicit} we can take the limits of all the divided
    differences.
    Then~\eqref{prop1} for $P_i$ follows
    from~\eqref{prop1} for $Q_i$. And the bound~\eqref{bnd} in the limiting case
    gives us $|P_s(z)| \le 2^{ns}\delta^n(D+1/2)^{n(s-1)} \le
    (2\delta)^n (2D+1)^{n(s-1)}$. The required bound
    follows.
\end{proof}

\section{Further work}

It is interesting to consider the quantity $E_n$ for disjoint compact sets
$I_1,\ldots,I_s\subset\mathbb C$. In what cases and in what terms can
one write an upper bound?

We obtained precise estimates  for $E_n$ in Case II; what about $E_n^*$?

The quantity $E_n(I_1,\dots, I_s)$  seems to be
interesting in itself. We do not know how to determine the extremal
pattern of the signs $y\in\{\pm1\}^s$ for given segments $I_j$.

One can  impose other additional requirements on the approximation polynomial  that may be important for some applications, e.g.
    $$
    \min(y_j,y_{j+1}) \le P(t) \le \max(y_j,y_{j+1})
    \quad\mbox{for $t$ between $I_j$, $I_{j+1}$}.
    $$

\end{document}